\documentclass[12pt,reqno, oneside]{amsart}
\usepackage{amssymb} 
\usepackage{pdfsync}  
\usepackage[text={6.3in,8.5in},centering]{geometry}
\usepackage{hyperref}
\usepackage{graphicx}
\usepackage{mathrsfs}
\usepackage{epigraph}

\newtheorem{thm}{Theorem}[section]

\newtheorem{lem}[thm]{Lemma}

\newcommand{\ve}{^{\bullet\raise 1.6pt\hbox{\vrule width 4pt height .8pt}}}
\newcommand{\edge}{^{\raise 1.6pt\hbox{\kern .3pt{\vrule width 5pt height .8pt}}}}
\newcommand{\vertx}{^{\bullet}}

\newcommand{\Tve}{t{^{\bullet\raise 1.6pt\hbox{\vrule width 4pt height .8pt}}}}
\newcommand{\Tedge}{t^{\raise 1.6pt\hbox{\kern .3pt{\vrule width 5pt height .8pt}}}}
\newcommand{\Tvertx}{t^{\bullet}}
\newcommand{\seqnum}[1]{\href{http://oeis.org/#1}{{#1}}}

\newcommand{\A}{\mathscr{A}}
\newcommand{\T}{\mathscr{T}}
\newcommand{\hi}{homeomorphically irreducible}

\usepackage{tikz}
\tikzstyle{vertex}=[circle, draw, inner sep=0pt, minimum size=8pt]
\newcommand{\vertex}{\node[vertex]}

\begin{document}
\title[Counting Homeomorphically Irreducible Trees]{Good Will Hunting's Problem:\\
Counting Homeomorphically Irreducible Trees}

\author{Ira M. Gessel$^*$}
\address{Department of Mathematics\\
   Brandeis University\\
   Waltham, MA 02453}
\email{gessel@brandeis.edu}

\date{\today\\ \null\kern 10pt$^*$Supported by a grant from the Simons Foundation (\#427060, Ira Gessel).}

\begin{abstract}
In the film \emph{Good Will Hunting},
the main character, a janitor at MIT named Will Hunting, attacks the problem of drawing all the homeomorphically irreducible trees with 10 vertices. Although the film suggests that this is a difficult problem, it is in fact quite easy. A much more interesting problem is counting homeomorphically irreducible trees with $n$ vertices for all $n$, a feat accomplished by Harary and Prins  in 1959. Here we give an exposition and simplification of Harary and Prins's result, introducing some of the fundamental ideas of graphical enumeration.
\end{abstract}

\maketitle

\setlength{\epigraphwidth}{.65\textwidth}
\epigraph{My colleagues and I have conferred, and there is a problem on the 
board right now that took us more than two years to prove. Let this be said: the gauntlet has been thrown down. But the faculty have answered, and answered with vigor.}{Professor Lambeau in the film \emph{Good Will Hunting}}

\section{Introduction.}
In the 1997 film \emph{Good Will Hunting}, the main character, a young janitor at MIT named Will Hunting, astounds MIT Professor Lambeau by tackling the problem ``Draw all the homeomorphically irreducible trees with $n=10$." In other words, draw all trees on 10 vertices with no vertices of degree 2, where isomorphic trees are considered the same. 

Like the other problems in the film, this one is at about the level of an undergraduate homework problem. Since there are only ten such trees (Will draws eight of them before he is interrupted), to solve the problem it is enough to find a systematic way of making sure that no trees have been omitted or repeated, and it's not hard to do this by considering the possible degrees of the vertices, breaking up the problem into small cases that are easily dealt with individually. (See, for example, Horv\'ath,  Kor\'andi and  Szab\'o \cite{hks}.)

A much more interesting problem is that of counting homeomorphically irreducible trees with $n$ vertices for \emph{all} $n$. Since we can't compute infinitely many numbers,  this requires some explanation.  
Ideally we would like a simple explicit formula, such as Cayley's formula $n^{n-2}$ for the number of \emph{labeled} trees on $n$ vertices. But there is unfortunately no such simple formula for counting unlabeled trees.
As an alternative, we might be satisfied with an efficient algorithm  for computing the numbers. 
(See Wilf \cite{wilf} for a discussion of what an answer is for a general enumeration problem; see also Pak \cite{pak}.)
But we can do better than this: we can  find a functional equation for the generating function for homeomorphically irreducible trees that enables us to count them quite efficiently.

We will first count unlabeled trees with no restrictions on the degrees of the vertices. The enumeration of these trees is one of the fundamental results of graphical enumeration, and our account is an introduction to the topic.  Once we understand how to count these trees, the modification needed for counting homeomorphically irreducible trees is easy. 

Unrestricted trees were first counted by Cayley \cite{cayley1, cayley2} and a simpler approach was given by Otter \cite{otter}.
We shall follow Otter's approach, with some variations and simplifications.
We count trees in two steps. First we count \emph{rooted} trees, using a  recursive decomposition that yields a functional equation for their generating function.  Rather than using P\'olya's theorem, as in some approaches, we use simple facts about enumeration of multisets to derive the functional equation. Then we reduce the enumeration of (unrooted) trees to that of rooted trees. To do this, Otter  introduced the method of  ``dissimilarity characteristic theorems,'' which relate the numbers of orbits of vertices and edges under the automorphism group of a tree. We use instead a ``dissymmetry theorem,'' introduced by Leroux  \cite{leroux} (see also Leroux and Miloudi \cite{l-m}), which is conceptually simpler and generalizes more easily. (See, e.g., Gainer-Dewar and Gessel \cite{gg}.)

Homeomorphically irreducible trees were first counted by Harary and Prins in 1959 \cite{hp}. An exposition of their work can also be found in the book of Harary and Palmer \cite[pp.~61--64]{graph-enum}, and another enumeration of homeomorphically irreducible trees, using the theory of combinatorial species, was given by Leroux and Miloudi \cite[Exemple 1]{l-m}.

\section{Graphs and trees}
\label{s-basics}

A \emph{graph} consists of a set of vertices, some pairs of which are joined by edges. Formally, a graph $G$ is an ordered pair $(V,E)$ where $V$ is a set, whose elements are called the \emph{vertices} of $G$, and $E$ is a set of $2$-element subsets of $V$, whose elements are called the \emph{edges} of $G$. 

Figure \ref{f-graph1}  
is a drawing of the graph 
$(\{1,2,3,4,5\}, \{\{1,2\},\{2,3\}, \{3,4\}, \{1,4\}, \{3,5\} \})$.

\begin {figure}[hbt]
\begin{tikzpicture}
	\vertex[fill] (1) at (0,0) [label=above:$1$] {};
	\vertex[fill] (2) at (1,1) [label=above:$2$] {};
	\vertex[fill] (3) at (2,0) [label=above:$3$] {};
	\vertex[fill] (4) at (1,-1) [label=above:$4$] {};
	\vertex[fill] (5) at (3.5,0) [label=above:$5$] {};
	
	\path
		(1) edge (2)
		(2) edge (3)
		(3) edge (4)
		(4) edge (1)
		(3) edge (5)
	;
\end{tikzpicture}
   \caption{A graph}
   \label{f-graph1}
\end{figure}

A \emph{walk} in a graph $G$ from vertex $v_1$ to vertex $v_k$ is a sequence $(v_1, v_2, \dots, v_k)$ of vertices of $G$ such that each $\{v_i, v_{i+1}\}$ is an edge of $G$. A walk is a \emph{path} if all its vertices are distinct. A walk $(v_1, v_2, \dots, v_k)$ is a \emph{cycle} if $k\ge3$, $v_1=v_k$, and $v_1, v_2,\dots, v_{k-1}$ are distinct.
For example, in the graph of Figure \ref{f-graph1}, $(1,2,1)$ is a walk that is neither a path nor a cycle, $(1,2,3,5)$ is a path, and $(1,2,3,4,1)$ is a cycle.

A graph $G$ is \emph{connected} if for any two vertices $u$ and $v$ in $G$, there is a path from $u$ to $v$.
A \emph{tree} is a connected graph with no cycles.

An important property of trees, which we will use in the proof of Theorem \ref{t-diss},  is that for any two vertices $u$ and $v$ in a tree, there is a unique path from $u$ to $v$.

The graph of Figure \ref{f-graph1} is connected, but is not a tree. Figure \ref{f-2trees} shows two trees.

The \emph{degree} of a vertex in a graph is the number of edges incident with the vertex. A graph is called \emph{\hi} (or \emph{series-reduced}) if it has no vertices of degree 2.

The term ``\hi" comes from the old-fashioned view of a graph as a topological space---a vertex of degree 2 together with its incident edges can be replaced with a single edge yielding a smaller ``homeomorphic" graph, so a graph that cannot be reduced in this way is \hi. See Figure \ref{f-2trees} for an example of this reduction.

\begin{figure}[hbt]
\centering
\begin{tikzpicture}
	\vertex[fill] (1) at (0,1) [label=above:$1$] {};
	\vertex[fill] (2) at (0,-1) [label=above:$2$] {};
	\vertex[fill] (3) at (1,0) [label=above:$3$] {};
	\vertex[fill] (4) at (2,0) [label=above:$4$] {};
	\vertex[fill] (5) at (3,0) [label=above:$5$] {};
	\vertex[fill] (6) at (4,1) [label=above:$6$] {};
	\vertex[fill] (7) at (4,-1) [label=above:$7$] {};

	\path
		(1) edge (3)
		(2) edge (3)
		(3) edge (4)
		(4) edge (5)
		(5) edge (6)
		(5) edge (7)
	;
\begin{scope}[shift={(2.3in,0pt)}] end 
	\vertex[fill] (1) at (0,1) [label=above:$1$] {};
	\vertex[fill] (2) at (0,-1) [label=above:$2$] {};
	\vertex[fill] (3) at (1,0) [label=above:$3$] {};
	\vertex[fill] (5) at (3,0) [label=above:$5$] {};
	\vertex[fill] (6) at (4,1) [label=above:$6$] {};
	\vertex[fill] (7) at (4,-1) [label=above:$7$] {};

	\path
		(1) edge (3)
		(2) edge (3)
		(3) edge (5)
		(5) edge (6)
		(5) edge (7)
	;
\end{scope}
\end{tikzpicture}
\caption{A homeomorphically reducible tree and its reduction}
\label{f-2trees}
\end{figure}
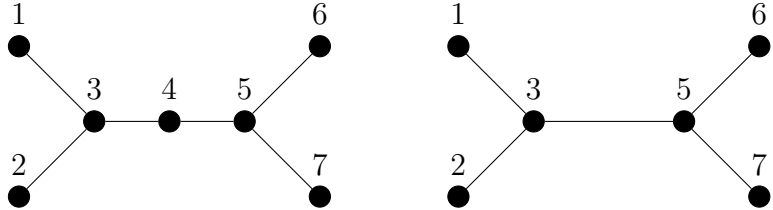

In counting graphs, we want to consider graphs which differ only in their labeling to be equivalent; in other words we want to count \emph{unlabeled graphs}. 
Informally, an unlabeled graph is a graph without labels on the vertices. To give a formal definition, we need to define isomorphism of graphs.

An \emph{isomorphism} from a graph $G_1=(V_1,E_1)$ to a graph $G_2=(V_2,E_2)$ is a bijection $\phi$ from $V_1$ to $V_2$ with the property that for vertices $u$ and $v$ in $V_1$,  $\{u,v\}\in E_1$ if and only if $\{\phi(u),\phi(v)\}\in E_2$. An isomorphism from a graph to itself is called an \emph{automorphism}.

If there is an isomorphism from $G_1$ to $G_2$ then we call $G_1$ and $G_2$ \emph{isomorphic}. Isomorphism is an equivalence relation, and the equivalence classes are called \emph{isomorphism classes}. Unlabeled graphs are formally just isomorphism classes of graphs, though we think of them intuitively as pictures like Figure \ref{f-unlabeled}, which represents the isomorphism class of the graph of Figure \ref{f-graph1}.
\begin {figure}[hbt]
\begin{tikzpicture}
	\vertex[fill] (1) at (0,0)  {};
	\vertex[fill] (2) at (1,1)  {};
	\vertex[fill] (3) at (2,0)  {};
	\vertex[fill] (4) at (1,-1)  {};
	\vertex[fill] (5) at (3.5,0) {};
	
	\path
		(1) edge (2)
		(2) edge (3)
		(3) edge (4)
		(4) edge (1)
		(3) edge (5)
	;
\end{tikzpicture}
   \caption{An unlabeled graph}
   \label{f-unlabeled}
\end{figure}
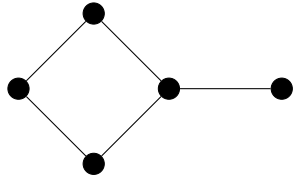

The only enumeration problems that we are concerned with here are of unlabeled trees, so we will usually omit the word ``unlabeled" when referring to enumeration. 

\subsection{Rooted trees.}
\label{s-rooted}
In order to count trees, we will need to count several types of rooted trees.

A \emph{vertex-rooted tree} is an ordered pair $(T,v)$ where $T$ is a tree and $v$ is a vertex of~$T$.
We think of a vertex-rooted tree as a tree with one vertex ``marked" as the root, as in the left half of Figure \ref{f-rtree-dec}.
An isomorphism from a vertex-rooted tree $(T_1, v_1)$ to a vertex-rooted tree $(T_2,v_2)$ is an isomorphism from $T_1$ to $T_2$ that takes $v_1$ to $v_2$. Unlabeled vertex-rooted trees may be defined as isomorphism classes of vertex-rooted trees. In Section \ref{s-unrooted} we will  discuss trees rooted at an edge or rooted at a vertex and incident edge, but for now we will refer to vertex-rooted trees simply as ``rooted trees."

There is a simple decomposition of rooted trees that allows them to be counted. Let $T$ be a rooted tree. If we remove the  edges incident with the root from $T$ then what remains is the root together with a set of trees.
Let us root each of these trees at the vertex that was adjacent to the root of $T$. Then $T$ can be recovered from its root and the set of rooted subtrees.
For example, Figure \ref{f-rtree-dec} shows a rooted tree and its decomposition.

\begin {figure}[hbt]
\begin{tikzpicture}
	\vertex[fill] (4) at (1,2) [label={[label distance=3pt] above:$4$}] {};
	\vertex[fill] (1) at (0,1) [label=left:$1$] {};
	\vertex[fill] (3) at (1,1) [label= right:$3$] {};
	\vertex[fill] (6) at (2,1) [label=right:$6$] {};
	\vertex[fill] (5) at (-.5,0) [label=left:$5$] {};
	\vertex[fill] (7) at (.5,0) [label=right:$7$] {};
	\vertex[fill] (2) at (2,0) [label= right:$2$] {};
	\draw [thin]  (4) circle [radius=0.22];
	\path
		(4) edge (1)
		(4) edge (3)
		(4) edge (6)
		(1) edge (5)
		(1) edge (7)
		(6) edge (2)
	;
\begin{scope}[shift={(2.3in,0pt)}] end 
	\vertex[fill] (4) at (1,2) [label={[label distance=3pt] above:$4$}] {};
	\vertex[fill] (1) at (0,1) [label=left:$1$] {};
	\vertex[fill] (3) at (1,1) [label= right:$3$] {};
	\vertex[fill] (6) at (2,1) [label=right:$6$] {};
	\vertex[fill] (5) at (-.5,0) [label=left:$5$] {};
	\vertex[fill] (7) at (.5,0) [label=right:$7$] {};
	\vertex[fill] (2) at (2,0) [label= right:$2$] {};
	\draw [thin]  (4) circle [radius=0.22];
	\draw [thin]  (4) circle [radius=0.27];
	\path
		(1) edge (5)
		(1) edge (7)
		(6) edge (2)
		;
	\draw [thin]  (1) circle [radius=0.22];
	\draw [thin]  (3) circle [radius=0.22];
	\draw [thin]  (6) circle [radius=0.22];
\end{scope}
\end{tikzpicture}
   \caption{A rooted tree and its decomposition}
   \label{f-rtree-dec}
\end{figure}
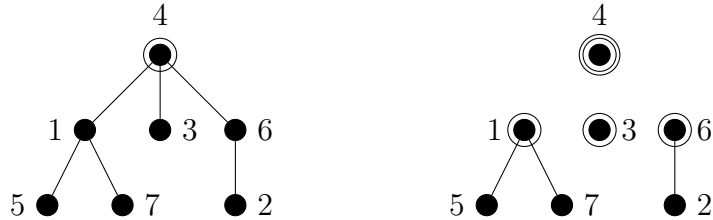

When we apply this decomposition to unlabeled trees, the set of rooted subtrees becomes a \emph{multiset} of unlabeled rooted trees, since these unlabeled rooted trees need not be distinct.  The reader may find it instructive to give a rigorous proof, using the definition of unlabeled trees in terms of isomorphisms, of the following intuitively obvious fact: There is a bijection from unlabeled rooted trees to pairs consisting of a single unlabeled vertex and a multiset of unlabeled rooted trees,  and this bijection preserves the total number of vertices.

\section{Weights, generating functions, and multisets.}
There is no simple formula for counting unlabeled trees.%
\footnote{There is a complicated explicit formula, as a sum over partitions, for counting unlabeled rooted trees due to Gilbert Labelle \cite[Corollary A2]{labelle}; see also Wagner \cite{wagner} and Constantineau and J. Labelle \cite{cl}. 
}
However, there are reasonably simple formulas for  \emph{generating functions} for counting unlabeled trees from which the numbers can be computed efficiently. The generating function 
for a sequence $u_0, u_1, u_2,\dots$ is  defined to be the infinite series 
$u_0+u_1x+u_2x^2+\cdots$, considered either as an analytic function or as a formal power series. (See Niven \cite{niven} for an exposition of the basic facts of formal power series.) We shall take a slightly different approach to generating functions: We have a set (often infinite) $\A$ of objects to be counted. Each element $a$ of $\A$ has a \emph{size} $s(a)$ associated to it. We then assign the \emph{weight} $x^{s(a)}$ to $a$, 
where $x$ is an indeterminate.\footnote{In other applications of generating functions more general weights may be used.} 
We assume that there are only finitely many elements of $\A$ of size $n$ for each $n$.
Then we define the \emph{generating function} $A(x)$ for $\A$ to be the  sum, as a formal power series, of the weights of all elements of $\A$:  $A(x)=\sum_{a\in \A} x^{s(a)}$.  Thus the coefficient of $x^n$ in $A(x)$ is the number of elements of $\A$ of size $n$.

To apply the decomposition of a rooted tree  into a root and a multiset of rooted trees,  we need to compute generating functions for multisets. Given a set $\A$ with a size function $s$, we
define the size of a multiset $\{a_1,a_2,\dots, a_k\}$ of elements of $\A$ to be 
$s(a_1)+s(a_2)+\cdots+s(a_k)$. Then the weight of the multiset $\{a_1,a_2,\dots, a_k\}$ is 
$x^{s(a_1)+s(a_2)+\cdots+s(a_k)}$, the product of the weights of its elements.
 We would like to find a formula for the generating function  for the set $M(\A)$ of all multisets of elements of $\A$ in terms of the generating function $A(x)$ for $\A$.
 
Let $h_n[A(x)]$ denote%
\footnote{This notation is taken from the theory of symmetric functions; see, e.g. \cite[p.~447, Definition A2.6]{EC2}.}
the generating function for multisets of $n$ elements chosen from $\A$, and let $h[A(x)] = \sum_{n=0}^\infty h_n[A(x)]$ be the generating function for all multisets of elements of $\A$. We will need formulas for $h_2[A(x)]$ and $h[A(x)]$. 

\begin{lem}
\label{l-h2}
The generating function for $2$-element multisets of elements of $\A$ is 
\begin{equation*}
h_2[A(x)]=\frac 12 \left(A(x)^2 + A(x^2)\right).
\end{equation*}
\end{lem}
\begin{proof}
First, $A(x)^2$ counts ordered pairs of elements of $\A$. So it counts each unordered pair of distinct elements twice and each unordered pair of the same element once. 
Next, $A(x^2)$ counts unordered pairs of the same element.

Therefore $A(x)^2 + A(x^2)$ counts every 2-element multiset of elements of $\A$ twice.
\end{proof}

\begin{lem}
\label{l-h}
The generating function for all multisets of elements of $\A$ is 
\begin{equation*}
h[A(x)] = \exp\biggl( \sum_{k=1}^\infty \frac{A(x^k)}{k} \biggr).
\end{equation*}
\end{lem}
\begin{proof}
If $\A$ has only a single element $a$, then  the generating function for  $M(\A)$ is
\begin{equation*}
1+x^{s(a)}+x^{2s(a)}+x^{3s(a)}+\cdots=\frac{1}{1-x^{s(a)}}.
\end{equation*}

So in the general case, the generating    function for $M(\A)$ is the (possibly infinite) product
\begin{equation}
\label{e-msetgf1}
\prod_{a\,\in \A} \frac{1}{1-x^{s(a)}}.
\end{equation}
The logarithm of \eqref{e-msetgf1} is
\begin{align*}
\log\prod_{a\,\in \A} \frac{1}{1-x^{s(a)}}
  &=\sum_{a\,\in \A}\log\frac{1}{1-x^{s(a)}}\\
  &=\sum_{a\,\in \A}\sum_{k=1}^\infty \frac{x^{ks(a)}}{k}\\
  &=\sum_{k=1}^\infty \frac{1}{k} \sum_{a\,\in \A }(x^k)^{s(a)}\\
  &=\sum_{k=1}^\infty \frac{1}{k}A(x^k),
  \end{align*}
and the desired formula follows by exponentiating.
\end{proof}

As a simple example of Lemmas \ref{l-h2}, we ask how many ways can we obtain a sum of $n$ when we roll two dice, where the dice are considered to be indistinguishable? For example we can obtain a sum of $4$ in two ways, as $3+1$ or $2+2$. Here the set $\A$ is $\{1,2,3,4,5,6\}$, the size of $i\in\A$ is just $i$, and $A(x) = x+x^2+x^3+x^4+x^5+x^6$. Then the generating function for unordered pairs of elements of $\A$ is 
\begin{align*}
\frac 12\bigl( \left(A(x)^2 + A(x^2)\right) &= \frac12 \bigl((x+x^2+x^3+x^4+x^5+x^6)^2 
   + (x^2+x^4+x^6+x^8+x^{10}+x^{12})\bigr)\\
&=  x^{2}+x^{3}+2 x^{4}+2 x^{5}+3 x^{6}+3 x^{7}+3 x^{8}+2 x^{9}+2 x^{10}+x^{11}+x^{12}.
\end{align*}
Similarly, by Lemma \ref{l-h}, the generating function for obtaining a sum of $n$ when rolling any number of dice is 
\begin{align*}
\exp\biggl( \sum_{k=1}^\infty \frac{A(x^k)}{k} \biggr)&=
\exp\biggl(\sum_{k=1}^\infty \frac1k(x^{k}+x^{2k}+x^{3k}+x^{4k}+x^{5k}+x^{6k})\biggr)\\
  &= 1+x +2 x^{2}+3 x^{3}+5 x^{4}+7 x^{5}+11 x^{6}+14 x^{7}+20 x^{8}+26 x^{9}+\cdots\end{align*}

\section{Counting  trees.}
\label{s-rootedgf}

To count trees, we first count rooted trees. We then relate the enumeration of unrooted trees to that of rooted trees.

\subsection{Rooted trees.}
Our decomposition of a rooted tree into a root together with a multiset of rooted trees gives the functional equation for the generating function $r(x)$ for rooted trees, where the size of a rooted tree is the number of vertices:
\begin{thm}
The generating function $r(x)$ for rooted trees satisfies
\begin{equation}
\label{e-R}
r(x) =x h[r(x)]= x \exp\biggl( \sum_{k=1}^\infty \frac{r(x^k)}{k} \biggr).
\end{equation}
\end{thm}
\begin{proof}
The decomposition of Section \ref{s-basics} gives a weight-preserving bijection from unlabeled rooted trees to pairs consisting of a single unlabeled vertex and  a multiset of unlabeled rooted trees.
The result then follows from Lemma \ref{l-h}, which gives the generating function for multisets of rooted trees in terms of $r(x)$.
\end{proof}
The easiest way to compute the coefficients of $r(x)$ from \eqref{e-R}, using a computer algebra system, is by successive substitution. 
Let
$r(x) = r_1 x+ r_2 x^2 + \cdots$
and let $r^{(n)}(x) = r_1 x+ r_2 x^2 +\cdots+r_n x^n$.
Then $r_{n+1}$ is the coefficient of $x^{n+1}$ in 
\[x\exp\biggl( \sum_{k=1}^n \frac{r(x^k)}{k} \biggr).\]
It is easy  to use this formula, inefficient though it is, with a computer algebra system to compute the numbers $r_n$. 

If we want our computer to do less work, we can derive a more efficient formula.
Suppose two generating functions $g(x)=\sum_{n=0}^\infty g_n x^n$ and $f(x)=\sum_{n=1}^\infty f_n x^n$ are related by $g(x)=e^{f(x)}$. 
Then $g'(x) = f'(x) e^{f(x)} = f'(x) g(x)$. Equating coefficients of $x^n$ gives the recurrence
\begin{equation}
\label{e-exprec}
(n+1)g_{n+1}=\sum_{k=0}^n (k+1) f_{k+1} g_{n-k}
\end{equation}
from which we can efficiently compute the coefficients of $g(x)$ in terms of those of~$f(x)$. 
Taking $g(x)=r(x)/x$ and $f(x) =\sum_{k=1}^\infty r(x^k)/k$, we get a recurrence for the numbers $r_n$ from \eqref{e-exprec} by setting 
$g_n=r_{n+1}$ and $f_n = \sum_{k\mid n}r_{n/k}/k$.

Either way, we find without difficulty that 
\begin{multline*}
r(x) =x+{x}^{2}+2{x}^{3}+4{x}^{4}+9{x}^{5}+20{x}^{6}+48{x}^{7}+115
{x}^{8}+286{x}^{9}\\
+719{x}^{10}+1842{x}^{11}+4766{x}^{12}+
12486{x}^{13}+32973{x}^{14}\\
+87811{x}^{15}+235381{x}^{16}+
634847{x}^{17}+1721159{x}^{18}\\
+4688676{x}^{19}+12826228{x}^{20
}+35221832{x}^{21}+\cdots\qquad\qquad
\end{multline*}
These numbers are sequence \seqnum{A000081} in the Online Encyclopedia of Integer Sequences (OEIS) \cite{oeis}.

\subsection{Unrooted trees.}
\label{s-unrooted}
To count unrooted trees, we find a formula that expresses the generating function for unrooted trees in terms of the generating function for rooted trees. We accomplish this through the \emph{dissymmetry theorem} of Pierre Leroux \cite{leroux, l-m}. (Otter \cite{otter} used a related, but somewhat different, result called the ``dissimilarity characteristic theorem.") 

Leroux's theorem relates unrooted trees to trees rooted at a vertex, trees rooted at an edge, and trees rooted at a vertex and incident edge.

\begin{thm}[The dissymmetry theorem] 
\label{t-diss}
Let $T$ be a \textup{(}labeled\textup{)} tree. Let $T\vertx$ be the set of rootings of $T$ at a vertex, let $T\edge$ be the set of rootings of $T$ at an edge, and let $T\ve$ be the set of rootings of $T$ at a vertex and incident edge. Then there is a bijection $\phi_T$ from 
$
T\vertx\cup T \edge
$
to
$ \{T\}\cup T\ve$
which is compatible with isomorphisms of trees; i.e., if $\alpha\colon T_1\to T_2$ is an isomorphism of trees, and we extend $\alpha$ in the natural way to rootings of $T_1$ and $T_2$, then $\alpha \circ \phi_{T_1}=\phi_{T_2}\circ\alpha$.

Thus $\phi_T$ gives a corresponding bijection for unlabeled trees.
\end{thm}

\begin{proof}
We start with the fact that every tree has either an edge or a vertex, called the \emph{center}, which is fixed by every automorphism of the tree.%
\footnote{All that we really need is that every tree has either a vertex or an edge that is fixed by every automorphism of the tree; the center as we define it is a convenient way of constructing such an edge or vertex.
}
The center may obtained by successively removing every leaf (vertex of degree 1) and its incident edge until only a single vertex or edge remains.
 (The center may also be characterized as the central vertex or edge on every longest path in the tree.)
 
Now suppose that $T^w$ is a tree $T$  rooted at a vertex or edge $w$. If $w$ is the center of $T$ then we define $\phi_T(T^w)$ to be $T$. Otherwise, there is a unique path from $w$ to the center. Then $\phi_T(T^w)$ is $T$ rooted at both $w$ and the next vertex or edge on this path (which might be the center).

It is easy to describe the inverse of $\phi_T$: $\phi_T^{-1}(T)$ is $T$ rooted at its center, and $\phi_T^{-1}$ of $T$ rooted at a vertex $v$ and incident edge $e$ is $T$ rooted at whichever of $v$ and $e$ is farther from the root. 

Since the center of a tree is fixed by every automorphism of the tree, it follows that any isomorphism is fixed of trees takes centers to centers, thus paths to the center to paths to the center.  Thus for any isomorphism $\alpha\colon T_1\to T_2$, we have $\alpha \circ \phi_{T_1}=\phi_{T_2}\circ\alpha$.
\end{proof}

For example, in the tree $T$ shown in Figure \ref{f-tree}, vertex 5 is the center. So $\phi_T$ takes $T$ rooted at 5 to $T$ unrooted, it takes $T$ rooted at 3 to $T$ rooted at $3$ and $\{3,6\}$, and it takes $T$ rooted at $\{5,6\}$ to $T$ rooted at $5$ and $\{5,6\}$.

\begin {figure}[hbt]
\begin{tikzpicture}
	\vertex[fill] (2) at (0,0) [label=left:$2$] {};
	\vertex[fill] (1) at (0,1) [label=left:$1$] {};
	\vertex[fill] (3) at (0,2) [label= left:$3$] {};
	\vertex[fill] (6) at (1,1) [label=above:$6$] {};
	\vertex[fill] (5) at (2,1) [label=above:$5$] {};
	\vertex[fill] (7) at (3,0) [label=right:$7$] {};
	\vertex[fill] (4) at (3,2) [label= above:$4$] {};
	\vertex[fill] (8) at (4,2) [label= above:$8$] {};
	\path
		(3) edge (6)
		(1) edge (6)
		(2) edge (6)
		(6) edge (5)
		(5) edge (4)
		(5) edge (7)
		(4) edge (8)
	;
\end{tikzpicture}
   \caption{A tree with center $5$}
   \label{f-tree}
\end{figure}
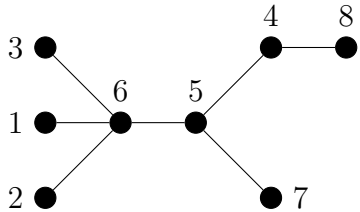

 Our main result on the enumeration of unrooted trees is a consequence of the dissymmetry theorem.
\begin{thm}
\label{t-unlabel}
Let $\T$ be a class of trees closed under isomorphism. Let $t(x)$ be the generating function for unlabeled $\T$-trees. Let $\Tvertx(x)$, $\Tedge(x)$, and $\Tve(x)$ be the generating functions for unlabeled $\T$-trees rooted at a vertex, at an edge, and at a vertex and incident edge. Then
\begin{equation*}
t(x) = \Tvertx(x) + \Tedge(x) -\Tve(x).
\end{equation*}
\end{thm}
\begin{proof}
Theorem \ref{t-diss} implies that
\begin{equation*}
t(x) + \Tve(x) = \Tvertx(x) + \Tedge(x).\qedhere
\end{equation*}
\end{proof}

Applying Corollary \ref{t-unlabel} to the case in which $\T$ is the class of all trees allows us to express the generating function for unrooted  trees in terms of the generating function $r(x)$ for rooted trees. 
\begin{thm}
\label{t-unr}
The generating function $u(x)$ for unrooted trees is given by 
\begin{equation*}
u(x) = r(x) -\frac 12 \left(r(x)^2 -r(x^2)\right),
\end{equation*}
where $r(x)$ is the generating for rooted trees, which satisfies
\begin{equation*}
r(x) =x\exp\biggl( \sum_{k=1}^\infty \frac{r(x^k)}{k} \biggr).
\end{equation*}
\end{thm}

\begin{proof}
In Theorem \ref{t-unlabel} we take $\T$ to be the class of all trees. We need to compute $\Tvertx(x)$, $\Tedge(x)$ and $\Tve(x)$.
We know that $\Tvertx(x)=r(x)$. If we remove the root edge  $\{v,w\}$ from a tree rooted at the edge $\{v,w\}$, we are left with a pair of trees. If we root these at $v$ and $w$, we obtain an unordered pair of rooted trees, and conversely, the original edge-rooted tree can be recovered from the unordered pair of rooted trees.  
This bijection is compatible with isomorphism so it applies to unlabeled trees, and thus we have
\[\Tedge(x) = h_2[r(x)] = \frac 12 \left(r(x)^2 + r(x^2)\right)\]
by Lemma \ref{l-h2}.
 
Similarly, if a tree is rooted at an edge $\{v,w\}$ and at the vertex $v$, then removing $\{v,w\}$ and rooting the two remaining trees at $v$ and $w$ gives an \emph{ordered} pair  $(T_1,T_2)$ of rooted trees, where $T_1$ is rooted at $v$ and $T_2$ is rooted at $w$, and the original vertex-edge-rooted tree can be recovered from this ordered pair.  Thus
\begin{equation*}
\Tve(x) = r(x)^2.
\end{equation*}

Then applying Theorem \ref{t-unlabel} gives
\begin{align*}
u(x) 
     &= r(x) + \frac 12 \left(r(x)^2 + r(x^2)\right)-r(x)^2 \\
     &=r(x) -\frac 12 \left(r(x)^2 -r(x^2)\right).\qedhere
\end{align*}
\end{proof}
By Theorem \ref{t-unr}  we can easily compute
\begin{multline*}
u(x) = x+{x}^{2}+{x}^{3}+2{x}^{4}+3{x}^{5}+6{x}^{6}+11{x}^{7}+23{x}
^{8}+47{x}^{9}\\+106{x}^{10}+235{x}^{11}+551{x}^{12}+1301{x}^{
13}+3159{x}^{14}\\+7741{x}^{15}+19320{x}^{16}+48629{x}^{17}+
123867{x}^{18}\\+317955{x}^{19}+823065{x}^{20}+2144505{x}^{21}+\cdots\qquad\qquad
\end{multline*}
These numbers are sequence \seqnum{A000055} in the OEIS.

\section{Homeomorphically irreducible trees.}
We can count \hi\ trees by modifying the approach of Section \ref{s-unrooted}. To apply Theorem \ref{t-unlabel}, we need to count \hi\ trees rooted at a vertex, at an edge, and at a vertex and incident edge.

In a vertex-rooted tree, a \emph{child} of a vertex $v$ is a vertex  adjacent to $v$ that is farther from the root than $v$. Thus if $v$ is the root, every vertex adjacent to $v$ is a child of $v$, so the number of  children of $v$ is the degree of $v$, but if $v$ is not the root then the number of children of $v$ is one less than the degree of $v$. 

The number of children of a vertex is important since in the decomposition for rooted trees described in Section \ref{s-rooted}, the number of trees into which a rooted tree is decomposed is the number of children of the root. 

To apply Theorem \ref{t-unlabel} we need to count three types of rooted \hi\ trees. These can all be counted in terms of an auxiliary class of rooted trees, which we call \emph{$S$-trees}. These are rooted trees in which no vertex has exactly one child. In other words, these are rooted trees in which the root does not have degree 1 and no non-root vertex has degree 2. 

Let $s(x)$ be the generating function for $S$-trees. Then the decomposition of Section \ref{s-rooted} gives%
\footnote{It is interesting to note that $s(x)$ can be expressed in terms of the series $r(x)$ of Section \ref{s-unrooted} that counts unrestricted rooted trees: Rewriting \eqref{e-S} as 
$s(x) =\bigl(x/(1+x)\bigr) h[s(x)],$
and comparing with \eqref{e-R}, we see that $s(x) = r\bigl(x/(1+x)\bigr)$.}
\begin{equation}
\label{e-S}
 s(x) = x\sum_{n\ne 1} h_n[s(x)] = x \bigl(h[s(x)] - s(x)\bigr).
\end{equation}

We can now express the generating function for \hi\ trees in terms of $s(x)$.

\begin{thm}
The generating function for  \hi\ trees is 
\begin{equation*}
(1+x)s(x)  + \frac12(1-x)s(x^2) -\frac12(1+x)s(x)^2,\end{equation*}
where $s(x)$ satisfies
\begin{equation*}
s(x) =x\biggl(\exp\biggl( \sum_{k=1}^\infty \frac{s(x^k)}{k} \biggr)\ -s(x)\biggr).
\end{equation*}
\end{thm}

\begin{proof}
We apply Theorem \ref{t-unlabel}, with $\T$ the class of \hi\ trees. We will express the generating functions $\Tvertx(x)$, $\Tedge(x)$, and $\Tve(x)$ for $\T$ in terms of $s(x)$.

A \hi\ tree rooted at a vertex decomposes into a root together with a multiset of some number other than 2 of $S$-trees. Thus
the generating function for \hi\ trees rooted at a vertex is 
\begin{align}
\Tvertx(x) &= x\sum_{n\ne 2} h_n[s(x)]=x\bigl(h[s(x)] -h_2[s(x)]\bigr)\notag\\
  &= (1+x)s(x) -xh_2[s(x)],\text{ by \eqref{e-S}.}\label{e-hiv}
\end{align}

Conveniently, joining two rooted $S$-trees with an edge between their roots converts the not-1 degree of the roots to not-2, so as in Section \ref{s-unrooted}, the generating function for \hi\ trees rooted at an edge is
\begin{equation}
\label{e-hie}
\Tedge(x)=h_2[s(x)]
\end{equation}
and the generating function\ for \hi\ trees rooted a vertex and incident edge is 
\begin{equation}
\label{e-hiev}
\Tve(x) = s(x)^2.
\end{equation}
By Theorem \ref{t-unlabel}, the generating function for \hi\ trees is  $\Tvertx(x) + \Tedge(x) -\Tve(x)$, where $\Tvertx(x)$, $\Tedge(x)$, and $\Tve(x)$ are given by 
\eqref{e-hiv}, \eqref{e-hie}, and \eqref{e-hiev}, and $s(x)$ satisfies \eqref{e-S}. The formulas given in the theorem then follow by applying Lemmas \ref{l-h2} and \ref{l-h} and simplifying.
\end{proof}

We can then compute as many terms as we want:
\begin{multline*}
T(x) =x+{x}^{2}+{x}^{4}+{x}^{5}+2{x}^{6}+2{x}^{7}+4{x}^{8}+5{x}^{9}+
10{x}^{10}
\\
+14{x}^{11}+26{x}^{12}+42{x}^{13}+78{x}^{14}+132
{x}^{15}
+249{x}^{16}\\
+445{x}^{17}
+842{x}^{18}+1561{x}^{19}+2988
{x}^{20}+5671{x}^{21}+\cdots
\end{multline*}
These coefficients are sequence \seqnum{A000014} in the OEIS.
Pictures of all \hi\ trees with at most twelve vertices can be found in Appendix II of Harary and Prins \cite{hp}.
Will's original problem corresponds to the coefficient of $x^{10}$, which is 10. 
A much more challenging problem for Will would have been to find the number of \hi\ trees with 100 vertices,  which is 76119905667088547333499833156.

\textbf{Acknowledgment.} I would like to thank Michael Levine for pointing out an error in an earlier version of this paper.

\bibliography{hunting}{}
\bibliographystyle{amsplain}

\end{document}